\documentclass[10pt]{amsart}
\UseRawInputEncoding
\usepackage{amssymb}
\usepackage{dsfont}
\usepackage{amscd}
\usepackage[mathscr]{euscript} 
\usepackage[all]{xy}
\usepackage{url}
\usepackage{stmaryrd}
\usepackage{comment}
\usepackage[retainorgcmds]{IEEEtrantools}
\usepackage{hyperref}
\usepackage{graphicx}
\usepackage{mathtools}
\usepackage{centernot}
\usepackage{marvosym}
\usepackage{tikz}
\usepackage{tkz-fct}
\usepackage{soul}
\usepackage{lineno}

\usepackage[T1]{fontenc}

\title[Stable formulas in ordered structures] {Stable formulas in ordered structures}

\author[D. M. HOFFMANN]{Daniel Max Hoffmann$^{\dagger}$}
\thanks{2010 \textit{Mathematics Subject Classification}. 
Primary 03C64; 
Secondary 03C45, 03C10}
\thanks{\textit{Key words and phrases}. real closed fields, ordered structures, stable formulas}
\thanks{$^{\dagger}$SDG. The first author is supported by
the Polish Natonal Agency for Academic Exchange and
 the National Science Centre (Narodowe Centrum Nauki, Poland) grants no. 2016/21/N/ST1/01465,
and 2015/19/B/ST1/01150.}
\address{$^{\dagger}$ Instytut Matematyki\\
Uniwersytet Warszawski\\
Warszawa\\
Poland
\newline \indent {\em and}\newline
\indent \hspace{1mm} Department of Mathematics \\ University of Notre Dame \\ Notre Dame \\ IN \\ USA}
\email{daniel.max.hoffmann@gmail.com}
\urladdr{https://sites.google.com/site/danielmaxhoffmann/home}

\author[C.M. Tran]{Chieu-Minh Tran$^{\ast}$}
\address{$^{\ast}$Department of Mathematics \\ University of Notre Dame \\ Notre Dame \\ IN \\ USA}
\email{mtran6@nd.edu}
\urladdr{https://faculty.math.illinois.edu/~mctran2/}

\author[J. Ye]{Jinhe Ye$^{\ddagger}$}
\address{$^{\ddagger}$ Institut de Math\'ematiques de Jussieu-Paris Rive Gauche}
\email{jinhe.ye@imj-prg.fr}
\urladdr{https://sites.google.com/view/vincentye}

 \DeclareMathOperator{\theo}{Th}

\DeclareMathOperator{\ddf}{DF}\DeclareMathOperator{\dcf}{DCF}\DeclareMathOperator{\scf}{SCF}

\DeclareMathOperator{\rcf}{RCF}

\DeclareMathOperator{\ess}{ess}
\DeclareMathOperator{\bd}{bd}
\DeclareMathOperator{\INTER}{int}
\DeclareMathOperator{\esb}{\bd^{\ess}}
\DeclareMathOperator{\esi}{\INTER^{\ess}}

\DeclareMathOperator{\Nn}{\mathbb{N}}

\DeclareMathOperator{\Rr}{\mathbb{R}}

\DeclareMathOperator{\Pp}{\mathbb{P}}

\newcommand{\sM}{\mathscr{M}}
\newcommand{\sC}{\mathscr{C}}
\newcommand{\RCF}{\mathrm{RCF}}
\newcommand{\monster}{\boldsymbol{\sM}}
\newcommand{\monsterset}{\boldsymbol{M}}
\newcommand{\ccl}{\mathrm{ccl}}

\newtheorem{theorem}{Theorem}[section]
\newtheorem{prop}[theorem]{Proposition}
\newtheorem{lemma}[theorem]{Lemma}

\newtheorem{fact}[theorem]{Fact}
\theoremstyle{definition}
\newtheorem{definition}[theorem]{Definition}
\newtheorem{example}[theorem]{Example}
\newtheorem{remark}[theorem]{Remark}

\theoremstyle{remark}
\newtheorem*{theorem*}{Theorem}
\newtheorem*{cor*}{Corollary}

\theoremstyle{definition}
\theoremstyle{definition}

\theoremstyle{definition}

\theoremstyle{remark}

\makeatletter
\providecommand*{\cupdot}{%
  \mathbin{%
    \mathpalette\@cupdot{}%
  }%
}
\newcommand*{\@cupdot}[2]{%
  \ooalign{%
    $\m@th#1\cup$\cr
    \sbox0{$#1\cup$}%
    \dimen@=\ht0 %
    \sbox0{$\m@th#1\cdot$}%
    \advance\dimen@ by -\ht0 %
    \dimen@=.5\dimen@
    \hidewidth\raise\dimen@\box0\hidewidth
  }%
}

\providecommand*{\bigcupdot}{%
  \mathop{%
    \vphantom{\bigcup}%
    \mathpalette\@bigcupdot{}%
  }%
}
\newcommand*{\@bigcupdot}[2]{%
  \ooalign{%
    $\m@th#1\bigcup$\cr
    \sbox0{$#1\bigcup$}%
    \dimen@=\ht0 %
    \advance\dimen@ by -\dp0 %
    \sbox0{\scalebox{2}{$\m@th#1\cdot$}}%
    \advance\dimen@ by -\ht0 %
    \dimen@=.5\dimen@
    \hidewidth\raise\dimen@\box0\hidewidth
  }%
}
\makeatother

\def\Ind#1#2{#1\setbox0=\hbox{$#1x$}\kern\wd0\hbox to 0pt{\hss$#1\mid$\hss}
\lower.9\ht0\hbox to 0pt{\hss$#1\smile$\hss}\kern\wd0}

\def\notind#1#2{#1\setbox0=\hbox{$#1x$}\kern\wd0
\hbox to 0pt{\mathchardef\nn=12854\hss$#1\nn$\kern1.4\wd0\hss}
\hbox to 0pt{\hss$#1\mid$\hss}\lower.9\ht0 \hbox to 0pt{\hss$#1\smile$\hss}\kern\wd0}

\begin{document}

\newcommand{\ov}{\overline}
\newcommand{\FC}{\mathfrak{C}}

\newcommand{\twoc}[3]{ {#1} \choose {{#2}|{#3}}}
\newcommand{\thrc}[4]{ {#1} \choose {{#2}|{#3}|{#4}}}
\newcommand{\Kk}{{\mathds{K}}}

\newcommand{\dlog}{\mathrm{ld}}
\newcommand{\ga}{\mathbb{G}_{\rm{a}}}
\newcommand{\gm}{\mathbb{G}_{\rm{m}}}
\newcommand{\gaf}{\widehat{\mathbb{G}}_{\rm{a}}}
\newcommand{\gmf}{\widehat{\mathbb{G}}_{\rm{m}}}
\newcommand{\gdf}{\mathfrak{g}-\ddf}
\newcommand{\gdcf}{\mathfrak{g}-\dcf}
\newcommand{\fdf}{F-\ddf}
\newcommand{\fdcf}{F-\dcf}
\newcommand{\mw}{\scf_{\text{MW},e}}

\newcommand{\BC}{{\mathbb C}}

\newcommand{\CC}{{\mathcal C}}
\newcommand{\CG}{{\mathcal G}}
\newcommand{\CK}{{\mathcal K}}
\newcommand{\CL}{{\mathcal L}}
\newcommand{\CN}{{\mathcal N}}
\newcommand{\CS}{{\mathcal S}}
\newcommand{\CU}{{\mathcal U}}
\newcommand{\CF}{{\mathcal F}}
\newcommand{\CP}{{\mathcal P}}
\newcommand{\CI}{{\mathcal I}}

\begin{abstract}
    We classify the stable formulas in the theory of Dense Linear Orders without endpoints, the stable formulas in the theory of Divisible Abelian Groups, and the stable formulas without parameters in the theory of Real Closed Fields. The third result, unexpectedly, requires the Hironaka's theorem on resolution of singularities.
\end{abstract}

\maketitle

\section{Introduction}
In recent years, we have seen rapid development of the neostability program which aims to extend the ideas of stability to other settings. Efforts have been made toward investigating weaker notions (NIP, simplicity, NSOP$_1$, NTP$_2$, etc), considering the stable components (stably dominated types, stables formulas, etc) in unstable theories, or a mix and match between these themes; see~\cite{tentzieg} for the relevant definitions (e.g. Section 8.2 for the stability related notions).
In this paper, we are interested in stable formulas\textemdash also called ``stable relations''\textemdash in unstable theories, in other words, the local stability of these theories. This is an old direction which nevertheless continues to hold relevance with recent applications in combinatorics (\cite{ArtemSergei16}, \cite{distal-reg}, \cite{ArtemSergei18},  \cite{GabeAnandCaroline18}, \cite{MalliarisShelah}). Stable formulas is related to thorn-forking \cite{thorn16} and is the subject of stable forking conjecture for simple theories (\cite{kim2001}, \cite{around}).
Surprisingly, not much attention have been paid to the down-to-earth problem of classifying stable formulas in frequently seen examples of unstable theory. Our goal here is to fill this gap for the most obvious unstable structures, those that involves an ordering.

We know that ordering gives us unstability. The example below provides us with a slightly more general situation where we have unstability, namely, the formula defines a ``large set'' with a  ``slope''. It also points out why ``large'' and having a ``slope'' is necessary.

\begin{example}\label{example:basic}
Consider a strictly increasing function $f:[0,1]_{\Rr}\to[0,1]_{\Rr}$ definable in the ordered field $(\Rr; +, \times)$
such that $f(0)=0$ and $f(1)=1$, and take $$D:=\{(a,b)\in [0,1]_{\Rr}^2 \; | \;b<f(a)\}.$$
We will construct a sequence $(a_i,b_i)_{i<\omega}$ such that $(a_i,b_j)\in D$ if and only if $i\leqslant j$.
We start with any $(a_1,b_1)\in D$. 
Take $0<a_2<a_1$ such that $f(a_2)<b_1$, and $0<b_2<f(a_2)$ and continue this way.
Such a sequence can be similarly produced if we replace $f$ by a decreasing function. On the other hand, it is easy to see that there is no such sequence $(a_i,b_j)$ such that $(a_i,b_j)\in D$ if and only if $i\leqslant j$ when $D$ is the graph of $f$ or the entire set $[0,1]_{\Rr}^2$. 
\end{example}

For the rest of the paper, let $T$ be either the theory of dense linear orderings in $L=\{<\}$ (DLO), the theory of divisible ordered abelian groups in $L=\{+,0,<\}$ (DOAG), or the theory of real closed fields in $L=\{+,-,\cdot,0,1,<\}$ (RCF), let $\sM$ be a model of $T$ with underlying set $M$,  let $\varphi(x;y)$ be an $L(M)$-formula, and let $\dim$ denotes the o-minimal dimension of $T$; see~\cite{Lou} for the basic definition and results. When we say that $\varphi(x;y)$ is stable, we implicitly assume that stability is with respect to the pair $(x;y)$.   
We say that $\varphi(x;y)$ is {\bf rectangular} if $\varphi(x;y)$ is $T$-equivalent to $\psi(x) \wedge \theta(y)$ with $\psi(x)$ and $\theta(y)$ being $L(M)$-formulas. It is easy to see that rectangular formulas are stable, and so are their boolean combinations. Propostion~\ref{prop:introduction}, which combines the later Propostion~\ref{prop:symdiff} and Proposition~\ref{prop:new.2.12},
tell us that Example~\ref{example:basic} essentially points us in the right direction. Note that the condition $\dim \varphi(\sM) = |x|+|y|$ is the  precise version of what we meant by ``large'', and the notion of rectangular formula makes precise the idea of ``having no slope''. 
 \begin{prop}\label{prop:introduction}
 Suppose $\varphi(x;y)$ is stable, and $\dim \varphi(\sM) = |x|+|y|$ . Then there is an $L(M)$-formula $\varphi'(x;y)$ which is a disjuntion of rectangular $L(M)$-formulas such that $\dim\big( \varphi(\sM) \triangle \varphi'(\sM)\big) < |x|+|y|$.
\end{prop}

Proposition~\ref{prop:introduction} is not sufficient for the purpose of classifying stable formulas in $T$ as it says nothing when $\dim\varphi(\sM)< |x|+|y|$. We say that $\varphi(x;y)$ is {\bf order-free} if  $\varphi(x;y)$ is equivalent over $T$ to quantifier-free $L(M)$-formulas which do not contain $<$.
Order-free formulas form another natural class of stables formulas. It is possible to have $\varphi(x;y)$ order-free with $\dim\varphi(\sM)< |x|+|y|$.
If we consider a conjunction of a rectangular formula and an order-free formula, we still obtain a stable formula. Formulas of this form are said to be {\bf special stable}. A natural guess would be that every stable $L(M)$-formula is equivalent over $T$ to a finite union of special stable formulas. 
In Theorem \ref{thm:DLO.DOAG.main},
we show this is the case
when the theory under consideration is either the theory of Dense Linear Orders without Endpoints or the theory of Divisible Ordered Abelian Groups:

\begin{theorem}
Suppose $T$ is either $\mathrm{DLO}$ or $\mathrm{DOAG}$ and $\varphi(x;y)$ is stable. Then $\varphi(x;y)$ is equivalent over $T$ to a disjunction of special stable $L(M)$-formulas.
\end{theorem}
We expect that a result in this line can be obtained for more general linear orderings in $L=\{<\}$ and linearly ordered abelian groups in $L=\{<, +, 0\}$. However, we do not address this question in this paper.

Now, let us move to theory of Real Closed Fields and start with the following example, where we can see that the above description breaks down.

\begin{example}\label{example:naive.wrong}
In the next figure, the part of the curve on the left above the dashed line is defined by the system of equations and inequalities on the right:

\begin{center}
\centering
\begin{minipage}{.5\textwidth}
\begin{center}
\includegraphics[scale=0.5]{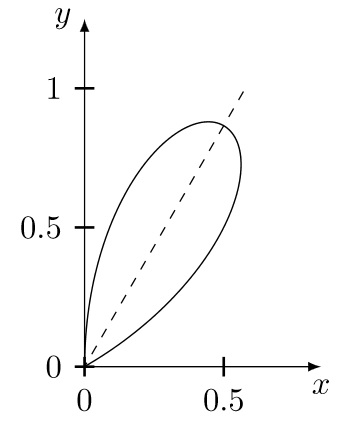} 
\end{center}
\end{minipage}%
\begin{minipage}{.5\textwidth}
  \centering
  \begin{IEEEeqnarray*}{rCl}
0 &=& x^{4}+2x^{2}y^{2}+y^{4}+x^{3}-xy^{2}; \\
0 &\leqslant& x; \\
0 &\leqslant& y; \\
\sqrt{3}x &\leqslant& y.
\end{IEEEeqnarray*}
\end{minipage}
\end{center}
Let $\varphi(x;y)$ with $|x| =|y|=1$ be the conjunction of the above equations and inequalities.  It is easy to check that $\varphi(x;y)$ is stable; in fact, every formula in two variables defining a one dimensional set is stable by cell decomposition. Note that the equation $x^{4}+2x^{2}y^{2}+y^{4}+x^{3}-xy^{2} $ defines an irreducible algebraic set of dimension $1$. 
Hence, if $\varphi(x;y)$ is a disjunction of special stable formulas, we can further arrange that each of these special stable formulas is a conjunction of $x^{4}+2x^{2}y^{2}+y^{4}+x^{3}-xy^{2} $ with a rectangular $L(M)$-formula. As $\mathrm{RCF}$ is o-minimal, we can arrange that each of these rectangular formulas defines a set of the form $I \times J$ where $I, J \subseteq \Rr$ are intervals. 
However, any set of the aforementioned form $I \times J$ containing the point $(0,0)$
will have to include a part of the curve below the dashed line $y=\sqrt{3}x$.
Thus, $\varphi(x;y)$ is not a disjunction of special stable $L(M)$-formulas.
\end{example}

In Example~\ref{example:naive.wrong}, the obstacle in expressing $\varphi(x;y)$ as a finite disjunction of special stable formulas comes from the singularity at $(0,0)$ of the curve in the picture. This brings us to the idea of using blowing up, or more precisely, resolution of singularities. It turns out that this is essentially the only obstruction when the stable formulas considered are parameter-free, and every such formula is equivalent over $\rcf$ to disjunctions of special stable formulas up to a certain kind of isomorphism. To be more precise, in Section \ref{sec:preliminaries}.3, we will define the notion of order-free isomorphism between a relation defined by a formula $\varphi(x;y)$ and a relation defined by a formula $\varphi'(x';y')$. This notion generalizes birational equivalence with a catch, namely, the division of variables must be respected. We obtain in Section~\ref{sec:stableRCF} our main result in this paper:

\begin{theorem}\label{thm:main.RCF0}
Suppose $T$ is $\rcf$. Then an $L$-formula $\varphi(x;y)$  is stable if and only if it is equivalent over $T$ to a 
disjunction of formulas order-free isomorphic to a special stable $L$-formula.
\end{theorem}
One could ask whether the statement of the above theorem  also holds for a stable formula with parameters, and we think the answer should be yes. Our current proof for Theorem~\ref{thm:main.RCF0}, in fact, goes through for the more general case when $\varphi(x;y)$ is a formula with parameters over an Archimedean subfield of $\sM$. 
However, the proof involves topological compactness. It is unclear if this technique transfers to the most general case when the infinite/non-Archimedean parameters occur in $\varphi(x;y)$. In a parallel direction, some natural subsequent questions could be the classification of stable formulas in $\Rr_\mathrm{exp}$ and $\Rr_\mathrm{an}$, multi-ordered fields (\cite{WillThesis}) or $\mathbb{Q}_p$.

\subsection*{Notations and conventions} Throughout $m$ and $n$ are in $\Nn = \{0, 1, \ldots\}$. We adopt the usual monster model gadgets: $\monster$ is a $\kappa$-saturated and strongly $\kappa$-homogeneous with underlying set $\monsterset$,  we refer to definable with parameters in $\monsterset$ as simply definable, parameter sets  $A$ and $B$ are assumed to be small with respect to $\monster$, and other models are assumed to be elementary submodels of $\monster$. Moreover, we identify a formula $\varphi$ with the set $D$ it defines in $\monster$, and work semantically with definable sets instead of formulas.

In the paper, we use the following convention about variables:
Throughout $x$ and $y$ are finite tuples of variables. Let $\monsterset^x$ denote the cartesian power of $\monsterset$ indexed by the variables in $x$.
By $(x,y)$ we denote the tuple $x$ extended by the tuple $y$.  With an eye on stability, we also use the usual notational convention $(x;y)$ to emphasize the division of variables. So when we write $D\subseteq \monsterset^{(x;y)}$, we mean $D\subseteq\monsterset^{(x,y)}$ and we have a fixed division of variables $(x;y)$.

\section{Preliminaries}\label{sec:preliminaries}
We continue to let $T$ be $\text{DLO}= \theo(\Rr,<)$,  $\text{DOAG}= \theo(\Rr,+,0,<)$, or $\RCF = \text{Th}(\Rr,+,-,\cdot,0,1,<)$. Let $\monster$ be a model of $T$ with underlying set $\monsterset$,  let $D \subseteq \monsterset^{(x;y)}$ be  $L(\monsterset)$-definable, and let $\dim$ denote the o-minimal dimension of $T$. Recall that $D \subseteq \monsterset^{(x;y)}$ is {\bf unstable} (with respect to the variable division $(x; y)$) if some     (every) formula defining $D$ is unstable. More explicitly, such $D$ is unstable if there is $(a_i,b_j)_{i < \omega, j< \omega}$ such that $(a_i, b_j) \in D$ if and only if $i \leqslant j$. We call such $(a_i,b_j)_{i < \omega, j< \omega}$ an {\bf unstable witness} of $D$. As $T$ is complete, classifying stable formulas over $T$ is the same as classifying stable sets  over $\monster$.

Toward classifying stable sets, we consider several classes of sets which are obviously stable.
We say that $D$ is \textbf{order-free definable} if it is defined by an order-free formula as in the introduction. This property is equivalent to being quantifier-free definable in the reduct of $\monster$ without the linear order. A typical example of an order-free definable set is an algebraic set in $\monster \models \RCF$. We say a function $f:X\to Y$ is order-free definable if the graph of $f$ is an order-free definable. Abusing notation, the restriction of $f$ to a subset $Z$ of $X$ is also called order-free definable.

A definable set is \textbf{rectangular} if it is defined by a rectangular formula as in the introduction. It is easy to see that $D$ is rectangular if and only if $D  = X \times Y$ with $X \subseteq \monsterset^x$ and $Y \subseteq \monsterset^y$ definable in $\monster$. Clearly, order-free definable sets and rectangular sets are stable. A special kind of stable sets which generalizes both order-free definable sets and rectangular set is given in the following definition.

\begin{definition}\label{def:special.stable}
We say that $D\subseteq \monsterset^{(x;y)}$ is a {\bf special stable set}  if 
$D$ is defined by a special stable formula (with respect to the division of variables $(x;y)$), equivalently, there exist definable sets
$X\subseteq \monsterset^{x}$ and $Y\subseteq \monsterset^{y}$,
and an order-free definable set $Z\subseteq \monsterset^{(x;y)}$ such that $D=Z\cap(X\times Y)$.
\end{definition}

The following easy fact will be later used in the proof of Theorem \ref{thm:main.RCF}.

\begin{lemma}\label{lem:Boolean.as.union}
Every Boolean combination of special stable sets is a finite union of special stable sets.
\end{lemma}

\begin{proof}
The intersection of special stable sets is again a special stable set, so it suffices to show that the complement of a special stable set is a finite union of special stable sets.
Consider a special set $D=Z\cap(X\times Y)$ where $X\subseteq\monsterset^{x}$ and $Y\subseteq\monsterset^{y}$ are definable, and $Z \subseteq \monsterset^{(x;y)}$ is order-free definable.
Note that 
$$D^{c}=Z^{c}\cup (X\times Y)^c=(Z^c\cap X\times Y)\cup (X\times Y)^c,$$
 $(X\times Y)^c=(X^c\times Y^c)\cup(X^c\times Y)\cup(X\times Y^c)$, and  $Z^c$ is order-free definable. The desired conclusion follows.
\end{proof}

It is well-known that the stability of formulas is preserved under taking Boolean combinations \cite[Lemma 2.1]{anandgeometric}, so Boolean combinations of special stable sets are stable. As it was noted in Example \ref{example:naive.wrong} from the introduction, not all stable sets are Boolean combinations of special stable sets.  We introduce the following notion to remedy this situation.

\begin{definition}\label{def:algebraic_morphism}
Suppose that $D\subseteq \monsterset^{(x;y)}$, $D'\subseteq \monsterset^{(x';y')}$, $\pi_x:\monsterset^{(x,y)}\to \monsterset^x$ and  $\pi_y:\monsterset^{(x,y)}\to \monsterset^y$ are the projection maps, and $\pi_{x'}$ and $\pi_{y'}$ are defined similarly. We say that a map $f:D'\to D$ is an {\bf order-free morphism} 
if there exist 
\begin{enumerate}
    \item a finite covering $(U_i)_{i\in I}$ of $D$ by order-free definable sets,
    \item a finite covering $(U'_i)_{i\in I}$ of $D'$ by order-free definable sets,
    \item a family of order-free definable maps $\big(f^i_x:\pi_{x'}(U'_i)\to\pi_x(U_i)\big)_{i\in I}$,
    \item a family of order-free definable maps $\big(f^i_y:\pi_{y'}(U'_i)\to\pi_y(U_i)\big)_{i\in I}$
\end{enumerate}
such that for each $i\in I$ we have 
that $f^i_x\times f^i_y(U'_i)= U_i$,
$f^i_x\times f^i_y(U'_i\cap D')=U_i\cap D$
and that
$f|_{U'_i\cap D'}=(f^i_x\times f^i_y)|_{U'_i\cap D'}$.
If, moreover,  the map $f^i_x\times f^i_y:U'_i\to U_i$ is a bijection for each $i\in I$, then we call $f$ an {\bf order-free isomorphism}. If there exists an order-free isomorphism between $D$ and $D'$, we say that $D$ is {\bf order-free isomorphic} to $D'$.
\end{definition}

Because of the preservation of stability under Boolean combinations, one can deduce the following easy fact:
if the definable sets $D_1,\ldots, D_n \subseteq \monsterset^{(x;y)}$ are stable and $D\subseteq  D_1\cup\ldots\cup D_n$ is definable, then $D$ is stable if and only if $D\cap D_i$ is stable for each $i\leqslant n$. This easy fact leads to the following, more important, observation:

\begin{lemma}\label{lemma:algebraic_iso_stable}
Suppose the definable sets $D\subseteq \monsterset^{(x;y)}$ and $D'\subseteq \monsterset^{(x';y')}$ are order-free isomorphic. If $D$ is stable, then $D'$ is stable.
\end{lemma}
 
\begin{proof}
Let $U'_i$, $U_i$, $f^i_x$ and $f^i_y$, for $i\in I$, be as in Definition \ref{def:algebraic_morphism}.
Suppose that $D'$ is unstable. From the fact preceding the lemma, we see that $D'\cap U'_k$ is unstable for some $k\in I$.
Let $(a'_i,b'_j)_{i,j<\omega}$ be an unstable witness of $D'\cap U'_k$.
By Ramsey's theorem \cite[Theorem 5.1.5]{tentzieg} there exists an infinite set $N\subseteq\omega$ such that for all $i,j\in N$ we have $(a'_i,b'_j)\in U'_k$ by the fact that $U_k$ is order-free.
Without loss of generality, we assume that $N=\omega$.

Consider $a_i:=f^k_x(a'_i)$, $b_j:=f^k_y(b'_j)$, where $i,j<\omega$. Because $(a'_i,b'_j)\in U'_k$ for all $i,j<\omega$,
we have that $$(a_i,b_j)=f^k_x\times f^k_y(a'_i,b'_j)\in U_k\text{ for all  } i,j<\omega.$$ If $i\leqslant j$ then $(a'_i,b'_j)\in D'\cap U'_k$
and thus $(a_i,b_j)\in D\cap U_k$.
If $i>j$, then $(a'_i,b'_j)\not\in D'\cap U'_k$. As $f^k_x\times f^k_y$ is injective over $U'_k$, we have that $(a_i,b_j) \not\in D\cap U_k$. Therefore, $(a_i,b_j)_{i,j<\omega}$ is an unstable witness of $D$.
\end{proof}

\section{Large stable sets}\label{sec:LSS}
Our strategy of classifying stable sets $D$ relies on an induction on $\dim(D)$. The goal of this section is to obtain a description of stable sets $D\subseteq \monsterset^{(x;y)}$ with $\dim(D)=|x|+|y|$. See Proposition~\ref{prop:symdiff} and \ref{prop:new.2.12} for such a description. Eventually, they will enable us to carry out the induction.

Next, we will obtain a description for large stable sets up to sets of smaller dimension. A definable set  $E$ is {\bf definably connected by codimension $2$} if for all definable $E' \subseteq E $ with $\dim( E \setminus E') \leq \dim E -2$, the set $E'$ is definable connected in the o-minimal structure $\monster$.
Before moving forward, let us recall a few basic facts from o-minimal structures, see \cite{Lou} for details on this subject.

Suppose $X$ is a definable subset of $\monsterset^x$. For $a \in \monsterset^x$, the {\bf local dimension} at $a$ of $X$  is defined as $$ \dim_a(X) = \inf\{ \dim( U \cap X) \mid U \text{ is a definable open subset of }  \monsterset^x\},$$  
where the ``$\dim(U\cap X)$'' refers to the o-minimal dimension and ``open'' refers to being open in the topology generated by open boxes defined by the ordering (here and in the rest of this paper an open box is just the Cartesian product of open intervals - similarly for a closed box). Suppose $E$ is an order-free definable subset of $\monsterset^x$ with nonempty $X \subseteq E$. As usual, the {\it closure} $\mathrm{cl}_E(X)$ of $X$ is the smallest closed subsets in $E$ containing $X$, the {\it interior} $\mathrm{int}_E(X)$ of $X$ in $E$ is the largest open subset of $E$ that $X$ contains, and the {\it boundary} $\mathrm{bd}_E(X)$is defined as $\mathrm{cl}_E(X)\setminus \mathrm{int}_E(X)$.
The {\bf essence} of $X$  in $E$ is the set 
$$\ess_E(X)=\{ a\in E \mid \dim_a(X) = \dim X \}.$$  The {\bf essential interior} of $X$ in $E$ is the set 
$$\esi_E(X)=\{a\in \ess_E(X)\mid \dim_a(E\setminus X)<\dim(X)\}.$$ 
The {\bf essential boundary} of $X$ in $E$ is the set 
$$\esb_E(X)=\{ a\in\
\ess_E(X) \mid \dim_a( E \setminus X) \geq \dim X  \}.$$ 
Note that $\ess(X)=\esi_E(X)\,\cup\,\esb_E(X)$. We will omit ``$E$'' if $E=\monsterset^n$ for some natural number $n$.

The following example illustrates the above notions. This was pointed out to us by the anonymous referee leading to a correction of a mistake in the earlier proof of Lemma~\ref{lem: boundarylemma}, where a wrong definition of essence was used.


\begin{example}
Let $X$ be the open unit ball in $\Rr^3$ with the equator, i.e.,
$$X= \{(x,y,z): x^2+y^2+z^2<1 \} \cup \{(x,y,z): z=0 ~x^2+y^2=1 \}$$ and $E=\Rr^3$. In this case, $\ess_E(X)$ is the unit closed ball. And $\esb_E(X)$ is the unit sphere and $\esi_E(X)$ is the open unit ball. In particular, $\esb_E(X)$ has dimension $2=3-1$, which is in accordance with Lemma~\ref{lem: boundarylemma}.
\end{example}

We collect some well known facts about o-minimal theories.
\begin{fact}\label{fact:omin.facts} 
\begin{enumerate}
    \item Every definable set can be definably triangulated in $T=\rcf$.
    \item Local dimension is definable in $T$.
    \item The closure/interior/boundary/essence/essential boundary/essential interior are definable in $T$.
    \item For a definable subset $X\subseteq \Rr^n$, the o-minimal dimension agrees with the Hausdorff dimension.
\end{enumerate}
\end{fact}
Statement (1) can be found in~\cite[Theorem~2.9, p.130]{Lou}. Statement (2) follows easily from~\cite[Chapter~4]{Lou}, and (3) is a consequence of (2). Item (4) can be proven easily using cell decomposition. In fact, a much more general phenomenon is true; see~\cite[Corollary 1.6]{PhilippMiller}.

The following fact is well-known. A proof can be found, for example, from~\cite{stackexchange}.
 
\begin{fact} \label{fact: stackexchange}
Every connected open subset  of $\Rr^x$ is topologically connected after removing a subset of Hausdorff codimension $2$.
\end{fact}

The next lemma gives us a model-theoretic counterpart of Fact~\ref{fact: stackexchange}.

\begin{lemma}\label{lem:incase}
Every definably connected open subset of $\monsterset^x$ is  definably connected by codimension $2$.
\end{lemma}
\begin{proof}
Let $U$ be a definably connected open subset of $\monsterset^x$. Suppose to the contrary.  Then we can obtain a definable $C \subseteq U$ with codimension $2$ in $U$ and a clopen subset $W$ of $U\setminus C$. Then we can obtain  $L$-formulas $\varphi(x,y)$, $\psi(x,y)$, and $\theta(x,y)$ such that there is $b \in \monsterset^y$ with $\varphi(x,b)$ defining $U$, 
$\psi(x,b)$ defining $C$,  and $\theta(x,b)$ defining $W$. Using the fact that $T$ is complete and dimensions are definable, we get $U'$, $C'$, and $W'$ in $\Rr$ with similar properties. This is a contradiction to the Fact~\ref{fact: stackexchange} and Fact~\ref{fact:omin.facts}(4). Hence, we have obtained the desired result.
\end{proof}

The next lemma is the key ingredient in proving Proposition~\ref{prop:symdiff}.

\begin{lemma}\label{lem: boundarylemma}
Suppose $X, E$ are subsets of $\monsterset^{x}$ with $X \subseteq E$, $\dim X = \dim E =n$ are definable and $E$ is order-free and definably connected by codimension $2$. 
Then either $$\dim_a(\esb_E(X))=n-1
\text{ for all } a \in \esb_E(X) \quad \text{or} \quad  \esb_E(X)=\emptyset.$$ Moreover, $\esb_E(X)=\emptyset$ if and only if $\dim (E \setminus X)< n$.
\end{lemma}

\begin{proof}
Note that DLO and DOAG are reducts of $\rcf$, and for definable sets in DLO and DOAG, their o-minimal dimension are the same as their o-minimal dimension when viewed as definable sets in $\rcf$. Hence, it suffices to consider the case where $T$ is $\rcf$. We will first show that if $\dim X = \dim  (E \setminus X) =n$, then $\dim(\esb_E(X))=n-1$. Note that $\esb_E(X)$ is a subset of the boundary of $X$ in  the ambient space $\monsterset^x$, so $\dim(\esb_E(X))\leq n-1$. Obtain a triangulation $\Delta$ of $E$ such that $X$ is the union of $\Delta_1 \subseteq \Delta$ and $E \setminus X$ is the union of $\Delta_2\subseteq \Delta$. It suffices to show that there is an open simplex in $\Delta_1$ and an open simplex in $\Delta_2$ sharing a common $(n-1)$-face, as points on this common $(n-1)$-face are elements of $\esb_E(X)$. Obtaining $\Delta'$, $\Delta'_1$, and $\Delta_2'$ from $\Delta$, $\Delta_1$, and $\Delta_2$ by taking only the simplices of dimension $\geq n-1$.
Set $E'$, $E'_1$, and $E'_2$ be the union of $\Delta$, $\Delta'_1$, and $\Delta'_2$. Let $a_1$ be in $E'_1$, and let $a_2$ be in $E'_2$. As $E$ is definably connected by codimension $2$ and semialgebraic, $E'$ is connected. Hence, there is a semialgebraic path $p: [0,1] \to E'$ with $p(0)=a_1$ and $p(1)=a_2$. Let $t_0 =  \sup\{t \mid p [0,t] \subseteq E'_1 \}$. Then $t_0$ lies on a simplex in $\Delta'$ which is a common face of a simplex in $\Delta'_1$ and a simplex in $\Delta'_2$. This common face is an element in $\Delta'$ and must have dimension $n-1$ as $\Delta'$ contains no element of dimension $<n-1$.

Suppose $\esb_E(X) \neq \emptyset$.  Let $a$ be an element of $\esb_E(X)$. Then we have $\dim X \geq \dim_a X \geq n$ and $ \dim (E \setminus X) \geq \dim_a (E \setminus X) \geq n$.  Hence, $\dim X = \dim ( E \setminus X) =n $. The same argument as in the preceding paragraph can be carried out replacing $E$ by $U$ and $X$ by $X \cap U$, where $U \subseteq E$ is any small open ball containing $a$. This gives us the stronger conclusion $\dim_a(\esb_E(X))=n-1$. If $\dim (E \setminus X) < n$, then $\esb_E(X) = \emptyset$.  If $\dim (E \setminus X)\geq  n$, then the argument of the preceding paragraph shows that $\dim(\esb_E(X))=n-1$. So we have obtained all the desired conclusions.
\end{proof}

We now prove the main result of this section.

\begin{prop}\label{prop:symdiff} Suppose $T$ is $\rcf$, $X\subseteq \monsterset^x$ and $Y \subseteq \monsterset^y$ are definable,  $D\subseteq X\times Y$ is stable with $\dim(X)= |x|$, $\dim(Y)= |y|$, and $\dim(D)=|x|+|y|$. Then there exists a finite family $(X_i,Y_i)_{i \in I}$ of semialgebraic sets $X_i\subseteq X$ and $Y_i\subseteq Y$ such that $\dim X_i=|x|$ and $\dim Y_i=|y|$ for each $i \in I$, and
$$\dim\big(D\triangle \bigcup\limits_{i \in I} (X_i\times Y_i)\big)< |x|+|y|.$$
\end{prop}

\begin{proof} 
In this proof, we will use the notion of tangent space and smoothness. This is defined as in~\cite{BCR}.
Removing sets of smaller dimension from $X$ and $Y$, we can arrange that $X$ and $Y$ are open subsets as subsets of $\monsterset^x$ and $\monsterset^y$. By decomposing $X$ and $Y$ into definably connected components, we may further assume that $X$ and $Y$ are definably connected. Hence, $X\times Y$ is a definably connected open subset of $\monsterset^{|x|+|y|}$. By Lemma~\ref{lem:incase}, $X \times Y$ is definably connected by codimension 2. Set $B=\esb_{X\times Y}(D)$. In the special case where $B$ is empty\textemdash by the last statement in  Lemma~\ref{lem: boundarylemma}\textemdash one can simply choose the family $(X_i, Y_i)_{i \in I}$ to consist of the single pair $(X, Y)$ and get $\dim(D\triangle (X\times Y))< |x|+|y|$. We will reduce the general situation to this special case.

Suppose $B=\esb_{X\times Y}(D)$ is non-empty. Lemma~\ref{lem: boundarylemma} then gives us  $\dim (B) = |x|+|y|-1$. Now, decompose $B$ into a finite disjoint union of definable Euclidean open subsets of smooth varieties in the sense of algebraic geometry (i.e. the dimension of the tangent space at every point is the dimension of the variety; see~\cite[Definition~3.3.4]{BCR}). 
Let $B^*$ be the set of points in the components with dimension $|x|+|y|-1$. 
Hence, 
$$\dim B^* = \dim B =|x|+|y|-1.$$ Suppose  $(a,b) \in X \times Y$ is a point in $B^*$. 
Let $T_{(a,b)}$ be the tangent space at $(a,b)$ of $B^*$. Then, $T_{(a,b)}$ also has dimension $|x|+|y|-1$. Let $T_{a}$ be the tangent space of $X$ at $a$, and let $T_b$ be the tangent space of $Y$ at $b$. Since $X,Y$ are open in $\monsterset^x$ and $\monsterset^y$,  $T_a$ is an isomorphic copy of the vector space $\monsterset^x$ over the underlying field of $\monster$, $T_b$ is an isomorphic copy of the vector space $\monsterset^y$ over the underlying field $\monster$, and $T_{(a,b)}$ is a hyperplane in $T_a \times T_b$.

We will show that either $T_{(a,b)} = S_{a} \times T_{b}$ where $S_a$ is a subset of $T_a$ with $\dim S_a = |x|-1$ or $T_{(a,b)} = T_{a}\times S_{b}$ where $S_b$ is a subset of $T_b$ with $\dim S_b =|y|-1$. 
Suppose it is neither of the above, then the projection maps from $T_{(a,b)}$ to $T_a$ and $T_b$ are surjective. 
We can then obtain a line $L_{(a,b)}$ in $T_{(a,b)}$, such that the projection of $L_{(a,b)}$ onto  $T_a$ and $T_b$ has dimension $1$. 
This translates to the existence of a curve $C$ in $B^*$ such that 
with $\pi_X C$ the projection of $C$ on $X$ and $\pi_Y C$ the projections of $C$ on $Y$, 
we have $ (\pi_X C\times \pi_Y C) \cap D$ and $(\pi_X C\times \pi_Y C)\setminus D$ is homeomorphic to the situation in Example \ref{example:basic}. More precisely, we obtain a continuous and increasing function  $f: [0,1]\to [0,1]$   with $f(0)=0$ and $f(1)=1$, together with definable homemorphisms $s: [0,1] \to \pi_X C $ and $t:[0,1] \to \pi_Y C$ such that the image of the graph of $f$ under $s\times t$ is $C$,
\[ (s\times t)\{ (c,d) \in [0,1]^2 :f(c)>d \}=   D\cap (\pi_X C\times \pi_Y C)
\]
and 
\[ (s\times t)\{ (c,d) \in [0,1]^2 :f(c)<d \}=   ((X \times Y)\setminus D)\cap (\pi_X C\times \pi_Y C).
\]
Then $s\times t$ maps the witness of unstability from Example \ref{example:basic} to an unstable witness in $D$, which is a contradiction.


Now, let $(a,b)$ range over $B^*$, and set
$$D_X=\{a \mid  \pi_X(T_{(a,b)})\neq T_a\}\quad \text{and} \quad D_Y=\{b|\, \pi_Y(T_{(a,b)})\neq T_b\}.$$ 
It is easy to check that $\dim D_X = |x|-1$ and $\dim D_Y =|y|-1$. Take a cell decomposition of $X$  such that $D_X$ is contained in the union of the cells of dimension $\leqslant |x|-1$, and a cell decomposition of $Y$ such that $D_Y$ is contained in the union of the cells of dimension $\leqslant |y|-1$.  Let $(X_i, Y_i)_{i \in I}$ be the collection of products of cells of dimension $|x|$ in the cell decomposition of $X$ and cells of dimension $|y|$ in the cell decomposition of $Y$. It is easy to see that $\esb_{X_i\times Y_i}(D\cap X_i\times Y_i) \subseteq B.$ Moreover, as $X_i \cap D_X = \empty$ and $Y_i \cap D_Y =\empty$ for all $i \in I$, so $X_i \times Y_i \cap B^* =\emptyset$. Therefore,
$$  \esb_{X_i\times Y_i}(D\cap X_i\times Y_i) \subseteq B \setminus B^*.$$
By general o-minimality knowledge, $\dim B \setminus B^* < \dim B = |x|+|y|-1.$ Thus, $ \esb_{X_i\times Y_i}(D\cap X_i\times Y_i) $ has dimension $< |x|+|y|-1$ by construction, and is therefore empty by  Lemma~\ref{lem: boundarylemma}. We reduced the situation to the special case at the beginning of this lemma.
\end{proof}  

In the remainder of this section, we will supose that $T$ is either DLO or DOAG and prove an analogue of Proposition~\ref{prop:symdiff}. In a model of RCF, we have two natural topologies, namely, the Zariski topology and Euclidean topology, with the former coarser than the latter. In the same fashion,  we can define on $\monsterset^x$  a topology coarser than the Euclidean topology. 
Let $\sC$ denote the collection of sets defined by lattice combinations (i.e. positive boolean combinations) of sets defined by subsets of $\monsterset^x$ defined by equations. As $T$ is complete, we can arrange that $\monster$ is a reduct of a model of $\RCF$. Recall that the Zariski topology in RCF is noetherian, so for any decreasing sequence ${(C_n)}_{n\in \Nn}$ of elements in $\sC$, there is $N \in \Nn$ such that $C_N=\bigcap_{n} C_{n}$. 
 Hence, $\sC$ is the collection of closed sets of certain noetherian topology which we call the {\bf linear topology} on $\monsterset^x$. As usual, an element $X$ of $\sC$ is {\bf irreducible} if $X$ cannot be written as a nontrivial union of two elements of $\sC$. It is then a standard fact about noetherian topology that every element of $\sC$ can be uniquely (up to permutation) decomposed into a finite union of irreducible elements of $\sC$ such that there is no containment between any two distinct elements.

\begin{remark}\label{remark:order-free.closure.form}
Suppose $x = (x_1, \ldots, x_m)$. If $T$ is DLO,  an irreducible closed subset of $\monsterset^x$ is the solution set of a system where each equation is of either the form $x_i=x_j$  or the form $x_i=c$ with $i$ and $j$ in $\{1, \ldots, m\}$ and $c \in \monsterset$.
If $T$ is $\mathrm{DOAG}$, an irreducible closed subset of $\monsterset^x$ is the solution set of a system consisting of equations of the form $k_1 x_1 + \cdots + k_m x_m =c$ with $k_i$ integers for $i \in \{1, \ldots, m\}$ and $c \in \monsterset$.
\end{remark}

With $T$ still either DLO or DOAG, we will define suitable versions of tangent spaces and smoothness in this setting. Suppose that $X\subseteq \monsterset^x$ is definable. Let $X^\ccl$ be the closure of $X$ with respect to the (coarser) linear topology on $\monsterset^x$, and let $\{ V_i\}_{i\leqslant n}$ lists all irreducible components of $X^\ccl$.
The {\bf tangent space} of $X$ at $a\in X$ is the smallest irreducible
closed set containing $\bigcup\{V_i\;|\;i\leqslant n,\,a\in V_i\}$. Remark~\ref{remark:order-free.closure.form} tells us that this a reasonable definition. 
Moreover, we say that $a\in X$ is {\bf smooth} if there is unique $V_i$ such that $a\in V_i$. The proof of Proposition \ref{prop:symdiff} with this new definition of smoothness and tangent space yields the Proposition~\ref{prop:new.2.12} below. Note that this is in accordance with the notion of smoothness in algebraic geometry; again, see~\cite[Definition~3.3.4]{BCR}. More precisely, in the case of DOAG, the $V_i$'s are all affine subspaces, hence algebraic varieties. All affine subspaces are in particular smooth varieties. Hence, a point $a\in \bigcup_i V_i$ is a smooth point if and only if there is a unique $i$ such that $a\in V_i$; a point lying on the intersection of two distinct irreducible components is not regular, hence not smooth.

\begin{prop}\label{prop:new.2.12}
Suppose $T$ is either $\mathrm{DLO}$ or $\mathrm{DOAG}$, $X\subseteq \monsterset^x$ and $Y \subseteq \monsterset^y$ are definable,  $D\subseteq X\times Y$ is stable with $\dim(X)= |x|$, $\dim(Y)= |y|$, and $\dim(D)=|x|+|y|$. Then there exists a finite family $(X_i,Y_i)_{i \in I}$ of definable sets $X_i\subseteq X$ and $Y_i\subseteq Y$ such that $\dim X_i=|x|$ and $\dim Y_i=|y|$ for each $i \in I$, and
$$\dim\big(D\triangle \bigcup\limits_{i \in I} X_i\times Y_i\big)< |x|+|y|.$$
\end{prop}

\section{Stable formulas in DLO and DOAG}
In this very short section, we classify the stable formulas in the case where $T$ is either DLO or DOAG.

\begin{theorem}\label{thm:DLO.DOAG.main} Suppose $T$ is either $\mathrm{DLO}$ or $\mathrm{DOAG}$ Then $D\subseteq \monsterset^{(x;y)}$ is stable if and only if $D$ is a finite union of special stable sets.
\end{theorem}

\begin{proof}
Let $Z$ be the closure of $D$ in the linear topology on $\monsterset^{(x;y)}$.  We argue by induction on the dimension of $D$. If $\dim(D)=0$, the conclusion is immediate. Since we are allowed to take finite unions in the theorem, we can arrange that $Z$ is irreducible in the linear topology on $\monsterset^{(x;y)}$. By Remark~\ref{remark:order-free.closure.form},  there are subtuples of variables $x'$ of $x$ and $y'$ of $y$ such that the projection map $\pi_{x',y'}: \monsterset^{(x;y)} \to \monsterset^{(x';y')} $  induces an isomorphism between $Z$ and $\monsterset^{(x';y')}$. From Proposition~\ref{prop:new.2.12}, we get  a finite set $I$, and  $X'_i$ and $Y'_i$ for each $i \in I$ such that $$\dim(\pi_{x',y'}(D)\triangle \bigcup_{i\in I} X'_i\times Y'_i)< \dim(Z) = \dim(D).$$
Set $X_i = \pi^{-1}_{x'}(X'_i)$ and $Y_i = \pi^{-1}_{y'}(Y'_i) $ where $\pi_{x'}: \monsterset^{x} \to \monsterset^{x'} $ and $\pi_{y'}: \monsterset^{y} \to \monsterset^{y'} $ are the projection maps. Then the dimension of the symmetric difference between $D$ and $Z \cap (\bigcup_{i \in I} (X_i \times Y_i))$ is strictly smaller than $\dim (D)$. Set $W_i$ to be the order-free closure of $\big(Z \cap (X_i\times Y_i)\big)\setminus D
$. Then for each $i \in I$, $ (Z \setminus W_i) \cap (X_i \times Y_i) $  is a special stable subset of $D$. Moreover, 
$$D\setminus \left(\bigcup_{i\in I} ((Z\setminus W_i)\cap (X_i\times Y_i)\right)$$ is stable and of smaller dimension. 
Hence, we can apply the inductive hypothesis and obtain the desired conclusion.
\end{proof}

One can observe that the above proof cannot be carried out when $T$ is $\rcf$ because if $Z$ is an irreducible variety, one cannot choose subtuples $x'$ of $x$ and $y'$ of $y$ such that the projection map $\pi_{x', y'}$ as defined in the proof induces an isomorphism.

\section{Stable formulas in RCF}\label{sec:stableRCF}

Finally, we are heading to the classification of stable formulas when $T$ is $\rcf$. With caveats, the strategy is the same as the proof of Theorem~\ref{thm:DLO.DOAG.main}, namely, taking projections to get largeness and then use Proposition \ref{prop:symdiff}.
This is quite similar to the proof of Theorem~\ref{thm:DLO.DOAG.main} except that we need to work harder to arrange for the projection maps to be bijective. To this end, we will need to decompose the original sets into finitely many disjoint pieces.
Hence, the need for topological compactness, which we get by passing to projective space through Lemma \ref{lemma:projectivization} and working with a fixed copy of $\Rr$ in the monster model. Another obstruction to finite decomposition comes from singularities, and this can be avoided by the use of Hironaka's resolution of singularities at the cost of obtaining a classification only up to order-free isomorphisms.

In this section, we will use real algebraic geometry in the classical sense (i.e. zeros of systems of polynomial); see \cite[Section~2, 3]{BCR} for the precise set up. We let
$\mathbb{P}^{x}$ denote of the $|x|$-dimensional projective space over $\monster$ corresponding to the tuple of variables $x$. We will identify $\monsterset^x$ in the usual way with a Zariski open subset of $\mathbb{P}^{x}$.

\begin{lemma}\label{lemma:projectivization}
Let $Z \subseteq \monsterset^{(x;y)}$ be an irreducible algebraic set. Then there are finite tuples $x'$ and $y'$ and an irreducible algebraic set $Z' \subseteq \monsterset^{(x';y')}$ satisfying the following properties:
\begin{enumerate}
    \item there are rational maps $f: \monsterset^{x'} \to \monsterset^x$ and $g: \monsterset^{y'} \to \monsterset^y$ such that $f\times g$ is a birational morphism from $Z'$ to $Z$. 
    \item the Zariski closure of $Z'$ in $ \mathbb{P}^{x'}\times \mathbb{P}^{y'} $ is smooth.
\end{enumerate}
\end{lemma}

\begin{proof}
View $Z$ as a subset of the projective space $\mathbb{P}^x\times\mathbb{P}^y$, and let $Z_{\text{pr}}$ be the Zariski closure of $Z$ in $\mathbb{P}^x\times\mathbb{P}^y$. Using Hironaka's theorem on resolution of singularity~\cite{Hironaka64}, we obtain a finite tuple of variables $z$ and an irreducible algebraic set $W_{\text{pr}}\subseteq\mathbb{P}^x\times\mathbb{P}^y\times\mathbb{P}^z$ such that the projection map from $\mathbb{P}^x\times\mathbb{P}^y\times\mathbb{P}^z$ to $\mathbb{P}^x\times\mathbb{P}^y$ induces a birational surjection from $W_{\text{pr}}$ to $Z_{\text{pr}}$. More precisely, the resolution of singularity is done using a sequence of blowups, so we can get the desired $z$; see~\cite[Theorem A]{BierMil} for details.
Choose  a new tuple $z'$ of variables with the same length as $z$. Copy $W_{\text{pr}}$ to $\mathbb{P}^{x}\times \mathbb{P}^{z} \times \mathbb{P}^{y} \times \mathbb{P}^{z'}$ to get $W'_{\text{pr}}$, or more precisely, let $W'_{\text{pr}}$ is the image of $W_{\text{pr}}$ under the map 
$$\mathbb{P}^x\times\mathbb{P}^y\times\mathbb{P}^z \to  \mathbb{P}^x\times\mathbb{P}^z\times\mathbb{P}^y\times \mathbb{P}^{z'}, (a,b, c)\mapsto (a,c, b, c).$$
Choose $x'$ and $y'$ such that $\mathbb{P}^x\times\mathbb{P}^z$ can be identified with a closed subset of $\mathbb{P}^{x'}$, and $\mathbb{P}^x\times\mathbb{P}^{z'}$ can be identified with a closed subset of $\mathbb{P}^{y'}$ via Segre embeddings. Identify $\mathbb{P}^x\times\mathbb{P}^z\times\mathbb{P}^y\times \mathbb{P}^{z'}$ as a closed subset of $\mathbb{P}^{x'}\times \mathbb{P}^{y'}$, and let $Z'_{\text{pr}}$ be the image of $W'_{\text{pr}}$ under this identification. Identify $\monsterset^{x'}$ with an affine piece of $\mathbb{P}^{x'}$ and $\monsterset^{y'}$ with an affine piece of $\mathbb{P}^{y'}$ in such a way that with  $Z' = Z'_{\text{pr}} \cap\monsterset^{x'}$ and $\monsterset^{y'} $, we have $Z'$ is dense in $Z'_{\text{pr}}$. Condition (2) is satisfied as the closure of $Z'$ in $\mathbb{P}^{x'} \times \mathbb{P}^{y'}$ is $Z'_{\text{pr}}$, which is smooth. By choosing suitable affine pieces, we get rational maps $f': \mathbb{P}^{x'} \to\mathbb{P}^{x} \times \mathbb{P}^{z}  $ and $g': \mathbb{P}^{y'} \to\mathbb{P}^{y} \times \mathbb{P}^{z'}$ such that $f' \times g'$ induces a birational morphism from $Z'_{\text{pr}}$ to $W'_{\text{pr}}$. Let $f=\pi_x \circ f'$ and $g= \pi_y \circ g'$ where $\pi_x$ and $\pi_y$ are projection onto $\mathbb{P}^{x}$ and $\mathbb{P}^{y}$. Then, $f$
and $g$ satisfy the condition specified in (1).
\end{proof}

For the next theorem, we are restricting our attention to a copy of the real numbers $\mathbb{R}$ living in the monster model $\monster$. The statement also holds if we replace $\Rr$ by an arbitary Archimedean subfield $K$ of $\monster$ as any such $K$ can be embedded into $\mathbb{R}$ using an automorphism of $\monster$. 

\begin{theorem}\label{thm:main.RCF}
A set $D\subseteq \monsterset^{(x;y)}$ definable over $\mathbb{R}$, is stable if and only if $D$ is a finite union of sets each order-free isomorphic over $\mathbb{R}$ to a special stable set defined over $\mathbb{R}$.
\end{theorem}

\begin{proof}
The backward direction follows from Lemma \ref{lemma:algebraic_iso_stable} and the well-known fact that  Boolean combination preserves stability.

Now, suppose that  $D \subseteq \monsterset^{(x; y)}$ is a  stable set defined over $\mathbb{R}$.  Let $D(\Rr)$ be the set of $\Rr$-points of $D$. By transfer principles, it suffices to show the forward direction of the theorem replacing $\monster$ with the field of real numbers, and $D$ with $D(\Rr)$. Moreover, we are not using the saturation of $\monster$ in the proof of the current theorem. Therefore, without loss of generality, we assume that $\monster=(\Rr; +, \times, <)$ and $D=D(\Rr)$.

Let $Z$ be the Zariski closure of $D$ in $\Rr^{(x; y)}$. 
If $Z$ has dimension $0$, the conclusion is immediate. By the fact that the stability is preserved under taking Boolean combinations, we can arrange that $Z$ is irreducible. Using induction on dimension of $Z$, we can assume that the statement is already proven for all stable sets with dimension smaller than $\dim(Z)$. This induction hypothesis, in particular, allows us to replace $Z$ with $Z'$ as in Lemma \ref{lemma:projectivization}, and so to arrange that the Zariski closure of $Z$ in $\mathbb{P}^{x} \times \mathbb{P}^{y}$ is smooth.

Let $(x'; y')$ range over the pairs where $x'$ and $y'$ are subtuples of $x$ and $y$ such that  $|x'|+|y'|=\dim(D)$, let $\pi_{x',y'}: \Rr^{(x;y)}\to\Rr^{(x';y')}$ be the projection map, and let
$U_{(x',y')}$ denote the maximal Zariski open subset of $Z$ such that the restriction of $\pi_{x',y'}$ to $U_{(x',y')}$ induces an isomorphism on the tangent spaces. Note that this set is indeed Zariski open, since it is defined by some non-vanishing conditions on minors of the Jacobian matrix. 
Since $Z$ is smooth, 
$$Z=\bigcup_{(x';y')} U_{(x',y')}.$$
For a fixed $(a,b) \in U_{(x',y')}$, by the definition of  $U_{(x',y')}$ and the inverse function theorem, the map $\pi_{x',y'}$ induces a homeomorphism from a neighborhood of $(a,b)$ in $Z$ onto its image with respect to the Euclidean topology. Therefore, we can obtain a set $B_{x',y'}^{a,b} \subseteq \Rr^{(x;y)}$ containing $(a,b)$ such that $B_{x',y'}^{a,b}$ is a cartersian product of open intervals in $\Rr$, and the restriction of $\pi_{x',y'}$ to $U_{(x',y')}\cap B_{x',y'}^{a,b}$ is one-to-one.

We consider first the case where $D$ is a subset of $[-1,1]^{(x;y)}$. As $(x'; y')$ ranges over the pairs specified earlier, and $(a,b)$ ranges over $U_{(x',y')}\cap [-1,1]^{(x;y)}$, the sets $U^{a,b}_{(x',y')}:=U_{(x',y')}\cap B^{a,b}_{x',y'}\cap [-1,1]^{(x;y)}$ form an open cover of $Z \cap [-1,1]^{(x;y)}$.
The intersection $Z \cap [-1,1]^{(x;y)}$ is closed, and hence compact. 
Hence, we can obtain a finite cover $(U_i)_{i\in I}$ of $Z \cap[-1,1]^{(x';y')}$ where each $U_i$ is of the form
 $U_{(x',y')}^{a,b}$with $(a,b)\in Z \cap [-1,1]^{(x;y)}$ 
 and $(x'; y')$ as specified earlier.

It is sufficient to show that for each $i\in I$, $D\cap U_i$ is a Boolean combination of sets algebraically isomorphic to special stable sets.
Fix $i \in I$ such that $D\cap U_i\neq\emptyset$. 
Note that $D\cap U_i$ is the intersection of the stable set $D$ with an algebraic set and a rectangular set, so $D\cap U_i$ is a stable subset of the special stable set $U_i$.
On the other hand, the restriction to $U_i$ of the projection map $\pi_{x',y'}$ is an isomorphism on tangent spaces and one-to-one. 
As a consequence, $\pi_{x',y'}(D\cap U_i)$ is stable in $(x';y')$ as any unstable witness in the image can be pulled back. 
Applying Proposition \ref{prop:symdiff}, we get that $\pi_{x',y'}(D\cap U_i)$ is up to a set of dimension smaller than $|x'|+|y'|=\dim(D) $ a finite union $E'$ of sets which are products of definable sets in $\Rr^{x'}$ and definable sets in $\Rr^{y'}$. Taking the preimage under $\pi_{x',y'}$, we learn that
the symmetric difference between $D\cap U_i$ and $E:=U_i\cap \pi_{x',y'}^{-1}(E')$ has dimension smaller than $|x'|+|y'|=\dim(D)$. 
Note that 
$$D\cap U_i= [E\setminus\big((U_i\setminus D)\cap E\big)]\;\cup\;[(U_i\setminus E)\cap D].$$
The two sets $(U_i\setminus D)\cap E$ and $(U_i\setminus E)\cap D$ have their union the symmetric difference between $D\cap U_i$ and $E$, so they are stable sets with dimension smaller than $\dim(D)$. By the inductive hypothesis, $(U_i\setminus D)\cap E$ and $(U_i\setminus E)\cap D$ are Boolean combinations of sets order-free isomorphic to special stable sets. As a consequence, $D\cap U_i$ is also of the desired form.

We reduce the general situation to the special case with $D \subseteq [-1,1]^{(x;y)}$. We construe $\Rr^{(x; y)}$ as an affine open subset of  $\Pp^x \times \Pp^y$. For a transition map $t$ from $\Rr^{(x; y)}$ to another affine open subset $\Rr_t^{(x; y)}$ of  $\Pp^x \times \Pp^y$, set $D_t$ to be $t( \text{Domain}(t) \cap D )$. As $t$ is an order-free isomorphism from $\Rr^{(x; y)}$ to $\Rr_t^{(x; y)}$, 
 $D_t$ as a subset of the $\Rr_t^{(x; y)}$ is again stable. Moreover, if $\dim D_t =\dim D$, then by our arrangement on $Z$, the Zariski closure of $ D_t$ is smooth.
For $(a,b) \in \Rr^{(x;y)}$, the interiors of $(a,b)+[-1,1]^{(x,y)}$ form an open cover of $\Rr^{(x;y)}$. Doing the same for the other affine pieces of $\Pp^x \times \Pp^y$ combine them together, we get a family of closed subset of $\Pp^x \times \Pp^y$ whose interiors form an open cover of $\Pp^x \times \Pp^y$. As $\Pp^x \times \Pp^y$ is compact, we obtain a finite subfamily $(V_j)_{j \in J}$ where $V_j$ is of the form $(a_j,b_j)+[-1,1]^{(x;y)}$ on the affine piece $\Rr_{t_j}^{(x; y)}$with the transition map $t_j$, and the interiors of the members of $(V_j)_{j \in J}$ form an open cover of $\Pp^x \times \Pp^y$. In particular, gives us that
$$  D = \bigcup_{j \in J} t_j^{-1} (D_{t_j} \cap V_j). $$ Hence, it suffices to show that for each $j \in J$, $ t_j^{-1} (D_{t_j} \cap V_j)$ is a finite union of special stable sets. When $\dim(D_t) = \dim(D)$,  $D_{t_j} \cap V_j$ is a finite union of special stable sets in $\Rr_{t_j}^{(x; y)}$ by the aforementioned special case. When $\dim(D_t) < \dim(D)$,  the desired conclusion follows from the induction hypothesis. As the transition map induces order-free isomorphism where it is defined, we get the desired conclusion.
\end{proof}

\subsection*{Acknowledgements} We would like to thank Lou van den Dries, Anand Pillay and Sergei Starchenko for several helpful discussions. We also thank the anonymous referee for many useful comments, and especially for pointing out a mistake in the earlier version.

\bibliographystyle{plain}
\bibliography{1nacfa2}

\end{document}